\documentclass[10pt]{amsart}
\usepackage{amsmath,amssymb,amsfonts, amscd,enumerate}
\usepackage{amsbsy}
\usepackage{amsthm}
\usepackage{amstext}
\usepackage{amsopn}
\usepackage{marvosym}
\usepackage{mathrsfs}
\usepackage{graphicx}
\usepackage{latexsym}
\usepackage[T1]{fontenc}
\usepackage[latin1]{inputenc}
\usepackage{color}
\usepackage{tikz}

\newcommand\blue[1]{\textcolor{blue}{#1}}

\newcommand{\Hmm}[1]{\leavevmode{\marginpar{\tiny%
$\hbox to 0mm{\hspace*{-0.5mm}$\leftarrow$\hss}%
\vcenter{\vrule depth 0.1mm height 0.1mm width \the\marginparwidth}%
\hbox to 0mm{\hss$\rightarrow$\hspace*{-0.5mm}}$\\\relax\raggedright #1}}}

\newtheorem{thm}{Theorem}[section]
\newtheorem{cor}[thm]{Corollary}

\newtheorem{lemma}[thm]{Lemma}
\newtheorem{pro}[thm]{Proposition}

\theoremstyle{definition}
\newtheorem*{defi}{Definition}
\newtheorem{eg}[thm]{Example}
\newtheorem{rem}[thm]{Remark}
\numberwithin{equation}{section}

\newcommand{\Z}{{\mathbb Z}}
\newcommand{\R}{{\mathbb R}}
\newcommand{\C}{{\mathbb C}}

\newcommand{\N}{{\mathbb N}}
\newcommand{\T}{{\mathbb T}}

\newcommand{\al}{{\alpha}}

\newcommand{\Cc}{{\mathcal C}}

\newcommand{\lm}{{\lambda}}
\newcommand{\ph}{{\varphi}}
\newcommand{\Deg}{{\mathrm{Deg}}}

\newcommand{\Hm}[1]{\leavevmode{\marginpar{\tiny%
$\hbox to 0mm{\hspace*{-0.5mm}$\leftarrow$\hss}%
\vcenter{\vrule depth 0.1mm height 0.1mm width \the\marginparwidth}%
\hbox to 0mm{\hss$\rightarrow$\hspace*{-0.5mm}}$\\\relax\raggedright
#1}}}

\begin{document}
\title[Magnetic sparseness and Schr\"odinger operators on  graphs]{Magnetic sparseness and Schr\"odinger operators on  graphs}

\author[M. Bonnefont]{Michel Bonnefont} \address{Michel Bonnefont,  Sylvain Gol\'enia, Institut de   Math\'ematiques de Bordeaux Universit\'e Bordeaux
351, cours de la Lib\'eration F-33405 Talence cedex, France}
\email{michel.bonnefont@math.u-bordeaux.fr \\
sylvain.golenia@math.u-bordeaux.fr}

\author[S. Gol\'enia]{Sylvain Gol\'enia}

\author[M. Keller]{Matthias Keller}
\address{Matthias Keller, Florentin M\"unch, Universit\"at Potsdam, Institut f\"ur Mathematik, 14476  Potsdam, Germany}
\email{matthias.keller@uni-potsdam.de\\chmuench@uni-potsdam.de}

\author[S. Liu]{Shiping Liu}
\address{Shiping Liu, School of Mathematical Sciences, University of Science and Technology of China, Hefei 230026, China}
\email{spliu@ustc.edu.cn}

\author[F. M\"unch]{Florentin M\"unch}

\subjclass[2000]{47A10, 34L20,05C63, 47B25, 47A63}
\keywords{discrete Laplacian, magnetic, locally finite graphs,
eigenvalues, asymptotic,  sparse, functional inequality}

\date{\today}

\begin{abstract} \noindent
We study magnetic Schr\"odinger operators on graphs. We extend the
notion of sparseness of graphs by including a magnetic quantity
called the frustration index. This notion of magnetic sparse turn
out to be equivalent to the fact that the form domain is an
$\ell^{2}$ space. As a consequence, we get criteria of discreteness
for the spectrum and eigenvalue asymptotics.
\end{abstract}

\maketitle
\section{Introduction}

Magnetic Schr\"odinger operators on graphs have been intensively
studied in recent years. The topics of research range from essential
self adjointness \cite{Mi, Mi2,CdVTHT, Gol} over Feynman-Kac-It\^{o}
formulas \cite{GKS,GMT} to spectral considerations
\cite{DM,GT,LLPP,LMP}.

We prove form estimates of magnetic Schr\"odinger operators in order
to give criteria for discreteness of the spectrum and to derive
eigenvalue asymptotics. Our approach extends the concept of
sparseness developed in \cite{BGK} by enhancing it with a magnetic
quantity called the frustration index. This index has its origin in
signed graphs \cite{Ha} and was used in \cite{LLPP} to prove
Cheeger estimates for magnetic operators on graphs (see also \cite{BSS}). Our concept of
magnetic sparseness includes the setting of \cite{LLPP} and allows
also for a structural understand of the recent considerations of
\cite{GT}.

Although magnetic operators were already considered in \cite{BGK}, the
magnetic potential was not used there. The results for magnetic
operators were deduced from the non-magnetic case using Kato's
inequality. By employing the magnetic field in our estimates a new
phenomena appears. While in \cite{BGK} sparseness is equivalent to
two sided estimates of the non-magnetic quadratic form, here
magnetic-sparseness is equivalent to lower bounds only. This can be
explained as follows: By a simple algebraic identity the lower bound
of a magnetic operator $H_{\theta}$ with magnetic potential $\theta$
is equivalent to an upper bound for the operator $H_{\theta+\pi}$
with magnetic potential shifted by $\pi$. Hence, two sided estimates
for $H_{\theta}$ are equivalent to two sided estimates for
$H_{\theta+\pi}$. So, by Kato's inequality it is conceivable that
lower bounds for $H_{0}$ imply lower bounds for $H_{\theta}$ for any
$\theta$ and, in particular, for $\theta=\pi$. However, Kato's
inequality in general fails to be an equality, so, in general one
can not infer a lower bound for $H_{0}$ from a lower bound of
$H_{\pi}$.

While we are mainly interested in the linear case of a quadratic form on $ \ell^{2} $, our techniques extend to the energy functional of the magnetic $ p $-Laplacian.

The paper is structured as follows. In the next section we introduce
the set up. In Section~\ref{s:result} we  present the main results and
 in Section~\ref{s:proof} we prove a magnetic co-area formula that is the core of our proof.
Furthermore, the main theorem is proven in this section. Finally, we
present examples in Section~\ref{s:examples}. In particular, in Section~\ref{s:prod}, we consider products of graphs and discuss how
the results of \cite{GT} can be embedded in our context.
Furthermore, discuss magnetic cycles in Section~\ref{s:cycle} and give a criterion for magnetic sparseness in terms of
subgraphs to apply this criterion to tessellations in
Section~\ref{s:subgraphs}.

\section{Set up and notation}\label{s:form}
\subsection{Magnetic forms} Let $X$ be a discrete set. We denote the complex valued
functions with finite support by $C_{c}(X)$. Any function
$m:X\to(0,\infty)$ extends to a measure of full support on $X$ and
we denote by $\ell^{p}(X,m)$, $ 1\leq p<\infty $ the complex Banach space with norm
$$\|f\|_{p} =\left (\sum_{x\in X}|f(x)|^{p} m(x)\right )^{1/p}.$$

A \emph{graph} over $X$ is a symmetric function $b:X\times X\to[0,\infty)$
with zero diagonal such that
\begin{align*}
    \sum_{y\in X}b(x,y)<\infty,\quad x\in X.
\end{align*}
A \emph{potential} is a function $q:X\to\R$ and a \emph{magnetic potential} is an
antisymmetric function $\theta:X\times X\to \R / 2\pi\Z$.

In the subsequent $m$ always denotes a measure, $b$ denotes a graph,
$q$ denotes a potential and $\theta$ denotes a magnetic potential
and we refer to the quadruple $(b,\theta,q,m)$ as the \emph{magnetic
graph}.

We call $\{x,y\}$ an
\emph{edge} of the graph if $b(x,y)>0$, $x,y\in X$.
We say a magnetic graph $(b,\theta,q,m)$ has \emph{standard edge
weights} if $b :X\times X \to \{0,1\}$.

We define
$Q^{(\mathrm{comp})}_{p,\theta}:=Q^{(\mathrm{comp})}_{p,b,\theta,q,m}:C_{c}(X)\to\R$
via
\begin{align*}
    Q^{(\mathrm{comp})}_{p,\theta}(\ph)=\frac{1}{2}\sum_{x,y\in
    X}b(x,y)|\ph(x)-e^{i\theta(x,y)}\ph(y)|^{p}+\sum_{x\in
    X}q(x)|\ph (x)|^{p}m(x),
\end{align*}
for $\ph\in C_{c}(X)$ and $p\in[1,\infty)$ and we write $Q^{(\mathrm{comp})}_{\theta} := Q^{(\mathrm{comp})}_{2,\theta}$ for $ p=2 $. Since the  focus of this paper is to study
the influence of the magnetic potential we highlight $\theta$ in
notation.

We have to bound the negative part
$q_{-}$ of $q$, where $q_{\pm}=(\pm q)\vee 0$. To this end we introduce the dual pairing for  functions  $ f\in C_{c}(X) $ and $g:X\to\C$
$$\langle g,f\rangle:=\sum_{x\in X}g(x)f(x)m(x).$$

Let $p\in[1,\infty)$, $b$, $\theta $ and $m$ be given. We say a potential $q$ is in
$\mathcal{K}_{\al}^{{p},\theta}$ for $\al >0$ if there is $C_{\al}\ge0$
such that
\begin{equation*}\label{eq:q-}
\langle q_{-},|f|^p\rangle \leq \al   Q^{(\mathrm{comp})}_{p,b,\theta,q_{+},m}{(f)}+C_{\al} \|f\|_p^p,\quad f\in C_{c}(X).
\end{equation*}
Clearly, $\mathcal{K}_{\al}^{p,\theta}\subseteq
\mathcal{K}_{\beta}^{p,\theta}$ if $\al\leq\beta$ and $
Q^{(\mathrm{comp})}_{p,b,\theta,q,m}$ is bounded from below if and only
if $q\in \mathcal{K}_{1}^{p,\theta}$. We define
\begin{align*}
   \mathcal{K}_{0^{+}}^{p,\theta}=\bigcap_{\al>0}\mathcal{K}_{\al}^{p,\theta}.
\end{align*}
Again we write $ \mathcal{K}_{\al}^{\theta}:=\mathcal{K}_{\al}^{2,\theta} $ and $ \mathcal{K}_{0^{+}}^{\theta}:= \mathcal{K}_{0^{+}}^{2,\theta}  $.
In \cite[Proposition 2.8]{GKS} it is shown that in the case $ p=2 $ the forms $
Q^{(\mathrm{comp})}_{2,b,\theta,q,m}$ are closable in $\ell^{2}(X,m)$
for any $q\in \mathcal{K}_{\al}^{\theta}$ with $\al\in(0,1)$ and we
denote the closure by $Q_{\theta}=Q_{2,b,\theta,q,m}$. Furthermore,
$D(Q_{2,b,\theta,q,m})=D(Q_{2,b,\theta,q_{+},m})$ and
\begin{align*}
    Q_{\theta}(f)=\frac{1}{2}\sum_{x,y\in
    X}b(x,y)|f(x)-e^{i\theta(x,y)}f(y)|^{2}+\sum_{x\in
    X}q(x)|f(x)|^{2}m(x),\quad f\in D(Q_{\theta}).
\end{align*}

{Note that, in the case  $q=0$, even if the value of $  Q_{\theta}(f)$ does not depend on $m$ for $f\in C_c(X)$,  its domain $D(Q_{\theta})$ does depend on $m$.}
Moreover, by \cite[Theorem~2.12]{GKS} the self adjoint operator
$H_{\theta}=H_{b,\theta,q,m}$ is a restriction of the \emph{formal
operator} $\mathcal{H}_{\theta}=\mathcal{H}_{b,\theta,q,m}$
\begin{align*}
    \mathcal{H}_{\theta}f(x)=\frac{1}{m(x)}\sum_{y\in
    X}b(x,y)(f(x)-e^{i\theta(x,y)}f(y))+q(x)f(x)
\end{align*}
for $f$ in $\mathcal{F}(X)=\{f:X\to\C\mid \sum_{y\in
X}b(x,y)|f(y)|<\infty,x\in X\}$.

We define the \emph{weighted vertex degrees} via
{
\begin{align*}
    \mathrm{Deg}&:X\to \R,\; x\mapsto \frac{1}{m(x)}\sum_{y\in
    X}b(x,y)+q(x),\\
    \mathrm{deg}&:X\to \R,\; x\mapsto \sum_{y\in
    X}b(x,y)+q(x)m(x).
\end{align*}

}

\subsection{Frustration indices}
Physically, two magnetic fields $\theta_1$ and $\theta_2$ act in the same way if $H_{\theta_1}$ and $H_{\theta_2}$ are  equivalent. From the perspective of the magnetic field this fact can be characterized in several equivalent ways, see e.g., \cite{HiSh, CdVTHT, LLPP, GT}. In this article we put forward the notion of frustration index.

The $ p $-\emph{frustration {index}},   $p\in[1,\infty)$, with respect to  $\theta$  of a finite set
$W\subseteq X$ is defined as
\begin{align*}
    \iota_{p,\theta}(W)&:=\min_{\tau:W\to\T}\frac{1}{2}\sum_{x,y\in
    W}b(x,y)|\tau(x)-e^{i\theta(x,y)}\tau(y)|^p,
\end{align*}
 where $\T{:=}\{z\in\C\mid |z|=1\}$ and the minimum is assumed by the compactness of $ \T^{W} $. { Note  that $\iota_{p, \theta}$ is  independent of $m$}. 
 Furthermore, we denote by
 $H_{\theta, W}$  the magnetic Laplacian associated to $(b|_{W\times W},\theta|_{W\times W},q|_W,m|_W)$ on $ \ell^{2}(W,m\vert_{W}) $ for finite $ W\subseteq X $.

 We summarize the basic properties of the frustration indices in the following proposition. 
 \begin{pro}\label{p:iota}
 Let  $(b,\theta,q,m)$ be a  magnetic graph, $ p\in[1,\infty) $ and $ W\subseteq X  $  finite. Then,
 \begin{itemize}
 	\item [(a)] $ \iota_{p,\theta}(W)\leq 2^{p-1}\iota_{1,\theta}(W) $ .
 	\item [(b)] $\iota_{p,\theta}(V)\leq
 	\iota_{p,\theta}(W)$ for  $V\subseteq W$.
 	\item [(c)]  $ \iota_{p,\theta}(W)= \sum_{\tilde W\in C(W)}  \iota_{p,\theta}(\tilde W), $  where the sum is taken over the connected components $ \tilde W\in C(W) $ of $ W $.
 	\item [(d)]  The following statements are equivalent:
 	\begin{itemize}
 		\item [(i)] $ \iota_{1, \theta}(W) = 0 $.
 		\item [(ii)] $ \iota_{p, \theta}(W) = 0 $.
 	    \item [(iii)] $ H_{\theta,W} $ is unitarily equivalent to  $ H_{0,W} $.
 		\end{itemize} If additionally $ q=0 $ then also the following statement is equivalent:
 		\begin{itemize}
 		\item [(iv)] $ \ker(H_{\theta, W})\neq \{0\} $.
 	\end{itemize}
 \end{itemize}
 \end{pro}
\begin{proof}Statement (a) follows  directly from the fact $|\tau(x)-e^{i\theta(x,y)}\tau(y)|\leq 2$
and (b) and (c)  are clear from the definition.

Let us turn to the equivalence in (d). The equivalence (i) $ \Longleftrightarrow $ (ii) is trivial since $ W $ is assumed to be finite. The equivalence (i) $ \Longleftrightarrow $ (iii)  can be seen using  \cite[eq.~(3.3)]{LLPP}.
To see the equivalence 	 (i) $ \Longleftrightarrow $ (iv) recall
that $\min \sigma (H_{\theta, W}) = \inf_{\|f\|_2=1}
\langle f, H_{\theta,W} f\rangle$. We see that
\[0\leq \min \sigma(H_{\theta,W}) \leq \iota_{2,\theta}(W)/ m(W)\leq  2\iota_{1,\theta}(W)/m(W).\]
This yields the implication (i) $
\Longrightarrow $ (iv). 	On the other hand assume (iv), i.e., $\ker(H_{\theta, W})\neq \{0\}$. Then, there exists a non-trivial $f: W\to
\C$  such that $H_{\theta,W} f=0$. In particular, $\langle
H_{\theta,W} f,f\rangle=0$ and $f(x)=e^{i \theta(x,y)} f(y)$ for all $x,y\in
W$ such that $b(x,y)>0$. W.l.o.g., we assume $W$ be connected. Then there is $\tau :W\to \T$ such that $\langle
H_{\theta,W} \tau,\tau\rangle=0$. In particular, $\iota_{1,\theta}(W)=0$.
\end{proof}

\subsection{Magnetic sparseness}
The boundary of a set $W\subseteq X$ is defined as
\begin{align*}
    \partial W{:=}\big( W\times (X\setminus W) \big) \cup \big(  (X\setminus W) \times
    W\big).
\end{align*}
To define quantities like the measure of the boundary or the
potential of a set we will use the convention that a non-negative function on a
discrete set extends to a  measure via additivity,  i.e., given a  set $A\subseteq X$ and $f : X\to [0,\infty)$ we let $f(A):=\sum_{x\in A} f(x)$.

We turn to the central notion of the paper.

\begin{defi}[Magnetic sparseness] Let $a,k\ge 0$ and $ p\in[1,\infty) $. We say the magnetic graph $(b,\theta,q,m)$ is
$(a,k)_{{p}}$-\emph{magnetic-sparse}
if
{
\begin{equation*}\label{e:defsparse}
    b(W\times W)\leq (1+a)\iota_{{p},\theta}(W)+a \big(\frac{1}{2}b(\partial
    W)+(q_{+}m)(W)\big) + k m(W)
\end{equation*}
}
for all finite $ W\subseteq X$. Furthermore, we say the magnetic
graph $(b,\theta,q,m)$ is $(a,k)_{{p}}$-\emph{bi-magnetic-sparse} if both
$(b,\theta,q,m)$ and $(b,\theta+\pi,q,m)$ are
$(a,k)_{{p}}$-{magnetic-sparse}, i.e.,
{
\begin{equation*}\label{e:defbisparse}
    b(W\times W)\leq (1+a)\big(\iota_{{p},\theta}(W)\wedge\iota_{{p},\theta+\pi}(W)\big)+a \big(\frac{1}{2}b(\partial
    W)+(q_{+}m)(W)\big) + k m(W)
\end{equation*}
}
for all finite $ W\subseteq X$, where $ \alpha\wedge \beta $ is the minimum of $ \al , \beta\in\R. $
\end{defi}

{
\begin{rem}\label{r:def}
The parameter $(1+a)$ in front of the frustation index may seem unnatural. We choose it because is the natural choice to prove a functional inequality for the Laplacian (see Theorem 3.1 below). However, it behaves a little less simple with respect to some  classical operation on graphs.  
\end{rem}
}

\begin{rem}\label{r:triv}
If $(b,\theta,q,m)$ is $(a,k)_{{p}}$-{magnetic-sparse}, Proposition~\ref{p:iota}~(a) ensures that $(b,\theta,q,m)$ is {$((1+a)2^{{p-1}}-1,k)_1$-{magnetic-sparse}}. We will prove below that the reciprocal is also true. Namely,
if a graph is  $(a,k)_1$-{magnetic-sparse} for some $a,k\geq 0$, it is also
{$(a(p),k(p))_{p}$}-{magnetic-sparse} for all $p\in[1,\infty)$ and some $a(p),k(p)\geq 0$, see Theorem \ref{main}.
\end{rem}

\begin{rem}
(a) Note that the $1/2$ in front of the boundary measure stems from
the fact that we want every boundary edge only once.

(b) In the case $q=0$, positivity of the isoperimetric quantity
\begin{align*}
    h_{{p},\theta}{:=}\inf_{W\subseteq X\mbox{\scriptsize finite }}
    \frac{\frac{1}{2}b(\partial W)+\iota_{{p},\theta}(W)}{\frac{1}{2}b(\partial W)+b(W\times W)}
\end{align*}
is equivalent to $(a,0)_p$-magnetic sparseness with some $a<\infty$, for $p \geq 1$.  {Moreover $a$ and $h_{p,\theta}$ are related by $a\leq \frac{1}{h_{p,\theta}}-1$} (whenever, the constant is $ a $ is chosen optimal then one even has equality).
The constant $h_{1, \theta}$ was considered in \cite{LLPP}. The constant $h_{2, \theta}$ appears in  the work of \cite{BSS}.
Note that for finite and connected $X$ and {$p\in[1,\infty)$}, we have  $h_{p,\theta}>0$ if and only if $\iota_{p,\theta}(X)>0$.
\end{rem}
\begin{eg}\label{e:finite}
{If $X$ is finite, then there is $k$ such that $(b, \theta, q, m)$ is $(0,k)_{{p}}$-bi-magnetic-sparse, for {$p\in[1,\infty)$}.}
\end{eg}

{The next proposition deals with magnetic bi-partite graphs. }
\begin{pro}\label{p:bi-partite}
Let  $(b,\theta,q,m)$ be a  magnetic {bi-partite} graph, $ p\in[1,\infty) $ and $ W\subseteq X  $  finite and $p\geq 1$. Then,
\[
 \iota_{p, \theta}(W) =  \iota_{p, \theta+\pi}(W).\] 
 In particular, a bi-partite graph
 $(b, \theta, q, m)$ is $(a,k)_{{p}}$-bi-magnetic-sparse if and only if
it is $(a,k)_{{p}}$-magnetic-sparse.
\end{pro}

\begin{proof}
Let $X=X_1 \cup X_2$ be a partition such that $b(x,y)=0$ if $(x,y)\in X_1 \times X_1$ or $(x,y)\in X_2 \times X_2$. For  $\tau:W\to \T$,
define $\tilde \tau: W\to \tau$ by 
\[
\tilde \tau(x)=\left\{\begin{array}{c c c}
\tau(x) & \textrm{ if } x\in X_1\\
-\tau (x)  & \textrm{ if } x\in X_2. 
\end{array}\right.
\]
The map $\tau \to \tilde \tau$ is clearly a bijection,  thus
\begin{align*}
 \iota_{p,\theta}(W)&=\min_{ \tilde \tau:W\to\T}\frac{1}{2}\sum_{x,y\in
    W}b(x,y)|\tilde \tau(x)-e^{i \theta(x,y)} \tilde \tau(y)|^p\\
    & = \min_{ \tau:W\to\T}\frac{1}{2}\sum_{x,y\in
    W}b(x,y)| \tau(x)+e^{i \theta(x,y)}\tau(y)|^p = \iota_{p,\theta+ \pi}(W).
  \end{align*}  
  The last point is clear from the definition.
\end{proof}

For further examples we refer the reader to Section~\ref{s:cycle} and Section~\ref{s:subgraphs}.

\section{Functional inequalities and magnetic-sparseness}
\subsection{Main results}\label{s:result}
The following theorem is the main result of the paper.
\begin{thm}\label{main}Let
 $(b,\theta,q,m)$ be a magnetic graph such that $q\in
\mathcal{K}_{\al}^{{p},\theta}$ for some $\al\in(0,1)$ and $p\in[1,\infty)$. The following assertions
are equivalent:
\begin{itemize}
  \item [(i)]  There are $a,k\geq 0$ such that the magnetic graph is $(a,k)_1$-magnetic-sparse.
\item [(i')]  There are $a',k'\geq 0$ such that the magnetic graph is $(a',k')_{{p}}$-magnetic-sparse.
  \item [(ii)] There are $\tilde a\in (0,1)$ and $\tilde k\geq 0$ such that for $f \in C_{c}(X)$
\begin{align*}
(1-\tilde a)\langle \Deg,|f|^p\rangle - \tilde k {\|f\|_p^p}  \leq Q_{{p,}\theta}(f)
\end{align*}
\end{itemize}
{Moreover, for $p=2$, i.e., $q\in \mathcal{K}_{\al}^{{2},\theta}$, the above assertions are also equivalent to:}
\begin{itemize}
  \item [(iii)]$D(Q_{\theta})= {\ell^{2}(X, \deg^{(+)} )\cap\ell^{2}(X,m)}$ for $ q\in\mathcal{K}_{\al}^{{2},\theta} $, where $ \deg^{(+)}=\deg+q_{-}$.
\end{itemize}

In \emph{(i')} the constant $a'$ can be chosen in dependence
of $a$ and $k$ given in \emph{(i)} and $\al $ such that
\begin{align*}
a'=\frac{2^{1-p}(p^2+1)^{\frac{p}{2}}(1+a)^p-\al}{1- \al}
\end{align*}
and $k'$ is a constant such that $k'=0$ if $k=0$ and
$q\ge0$.

In \emph{(ii)} the constant $\tilde a$ can be chosen in dependence
of $a$ and $k$ given in \emph{(i)} and $\al $ such that
\begin{align*}
\tilde a=1-{\frac{1- \al}{2^{1-p}(p^2+1)^{\frac{p}{2}}(1+a)^p-\al}}
\end{align*}
and $\tilde k$ is a constant such that $\tilde k=0$ if $k=0$ and
$q\ge0$.
\end{thm}

\begin{rem}In the case $ q=0 $ and $ k=0 $ in (i) or (i'), the inequality in (ii) and the equality in (iii) hold for all measures $ m $.
\end{rem}
In the non-magnetic case and $ p=2 $, the lower  bound in (ii) is
equivalent to a corresponding upper bound.  Via the equality
\begin{equation*}\label{e:pi}
\mathcal{H}_{\theta}=2\Deg-\mathcal{H}_{\theta+\pi},
\end{equation*}
which holds on $ \mathcal{F}(X) $ and the Green's formula
\begin{align*}
Q_{\theta}(f,g)=\sum_{x\in X}\mathcal{H}_{\theta}f(x)g(x)m(x),\quad
f,g\in C_{c}(X),
\end{align*}
statement (ii) in the theorem above can be seen to be equivalent to
\begin{itemize}
	\item [(ii')] $Q_{\theta+\pi}(f)\leq (1+\tilde a)\langle \mathrm{Deg},|f|^{2}\rangle +\tilde k\|f\|_{2}^{2}.
	$
\end{itemize}	

From Proposition~{\ref{ex:notbi}} and  Example~\ref{eg:not bi sparse} in the next section,  it can be seen
that there exist magnetic sparse graphs which are not
bi-magnetic-sparse. That is, a
corresponding upper bound for $Q_{\theta}$ does not necessarily hold
even if $Q_\theta$ is lower bounded. Nevertheless, for
bi-magnetic-sparse graphs we can obtain lower and upper bounds in the case $ p=2 $.

\begin{cor}\label{c:main}
Let $(b,\theta,q,m)$ be a magnetic graph such that $q\in
\mathcal{K}_{\al}^{{\theta}}$ for some $\al\in(0,1)$, i.e., $ p=2 $. Let $r\geq 1$. The following assertions
are equivalent:
\begin{itemize}
  \item [(i)] There are $a,k\geq 0$ such that the magnetic graph
  is $(a,k)_1$-bi-magnetic-sparse.
  \item [(i')] There are $a',k'\geq 0$ such that the magnetic graph
  is $(a',k')_{{r}}$-bi-magnetic-sparse.
  \item [(ii)] There are $\tilde a\in (0,1)$ and $\tilde k\geq 0$ such that for $f \in C_{c}(X)$
\begin{align*}
	(1-\tilde a)\langle \mathrm{Deg},|f|^{2}\rangle- \tilde k\|f\|_{2}^{2} \leq Q_{\theta}(f)\leq (1+\tilde
	a)  \langle \mathrm{Deg},|f|^{2}\rangle +\tilde k\|f\|_{2}^{2}.
\end{align*}
  \item [(iii)]$D(Q_{\theta})= D(Q_{\theta+\pi})= \ell^{2}(X, \deg^{(+)} \cap\ell^{2}(X,m)$,  where $ \deg^{(+)}=\deg+q_{-}$.
\end{itemize}
\end{cor}
\begin{proof}Apply Theorem~\ref{main}
with $\theta$ and $\theta+\pi$ and condition (ii') discussed above.
\end{proof}

\begin{rem}
	The characterization of bi-magnetic sparseness via the upper bound only works for $p=2$ since $Q_{p,\theta} (f)+ Q_{p,\theta+\pi} (f)= 2\langle\Deg,|f|^p\rangle$ for all $ f\in C_{c}(X) $ if and only if $p=2$.
	For $p\neq 2$, we only get estimates between $Q_{p,\theta}(f) + Q_{p,\theta+\pi}(f)$ and 	$\langle\Deg,|f|^p\rangle$  for all $ f\in C_{c}(X) $.	
\end{rem}
Next, we turn to the corresponding eigenvalue asymptotics in the case $ p=2 $. When $H_{\theta}$ has purely discrete spectrum, i.e., the spectrum of $ H_{\theta} $ consists of discrete eigenvalues with finite multiplicity, we denote
the eigenvalues  counted with multiplicity in increasing order by
$\lm_{n}$, $n\ge0$. Furthermore, set
\begin{align*}
D_{\infty}:=\sup_{K\subseteq X \mbox{\scriptsize\, finite}}
  \inf_{x\in X\setminus K} \Deg(x),
\end{align*}
and, in the case $ D_{\infty}=\infty $, order the vertices $\N_{0}\to X$, $n\mapsto  x_{n} $
bijectively such that
\begin{align*}
    D_{n}:=\mathrm{Deg}(x_{n})\leq \mathrm{Deg}(x_{n+1}).
\end{align*}


\begin{cor}\label{c:ev}If the magnetic graph $(b,\theta,q,m)$ is
	$(a,k)_{{1}}$-magnetic sparse, and if { $q\in \mathcal{K} _\alpha^\theta$} for some
	$\al\in (0,1)$. Then the following statements are equivalent:
	\begin{itemize}
		\item[(i)] $H_{\theta}$ has purely discrete spectrum.
		\item[(ii)] $D_{\infty}=\infty$.
	\end{itemize}
	In this case we have for the eigenvalues $ \lm_{n} $ of $ H_{\theta} $
	\begin{align*}
	\frac{2(1-\alpha)}{5(1+a)^{2}-\alpha}\leq \liminf_{n\to\infty}\frac{\lm_{n}}{D_{n}}
	\leq \limsup_{n\to\infty}\frac{\lm_{n}}{D_{n}}.
	\end{align*}
\end{cor}

\begin{rem}\label{r:0+}
	In the case where $q\in \mathcal{K}_{0^{+}}$, one can take $\alpha=0$ in Corollary \ref{c:ev} and Remark \ref{r:0+}.
\end{rem}

\begin{rem}
 If the case where    $(b,\theta,q,m)$ is actually
	$(a,k)_{{1}}$-bi-magnetic sparse, using Corollary \ref{c:main}  one gets:
	\begin{align*}
	\frac{2(1-\alpha)}{5(1+a)^{2}-\alpha}\leq \liminf_{n\to\infty}\frac{\lm_{n}}{D_{n}}
	\leq \limsup_{n\to\infty}\frac{\lm_{n}}{D_{n}}\leq 2 -  \frac{2(1-\alpha)}{5(1+a)^{2}-\alpha}.
	\end{align*}
\end{rem}

\begin{rem}
We introduce  $d(x):=\frac{1}{m(x)} \sum_{y\in X} b(x,y)$ for $x\in X$. One has $\Deg(x)= d(x)+q(x)$  and via the inequality $|z+w|^{2}\leq 2(|z|^{2}+|w|^{2})$, $z,w\in \C$,  one gets the estimate  on $C_{c}(X)\subseteq \ell^2(X,m)$
		\[
		Q_{\theta}\leq 2 d+ q=\left( \frac{2 +\frac{q}{d}}  {1 +\frac{q}{d}}\right) \Deg.
		\]
	Setting \[
		l:=\sup_{K\subseteq X \mbox{\scriptsize\, finite}}
		\inf_{x\in X\setminus K} \frac{ q(x)}{d(x)}
		\]
	we get by  the Min-Max-Principle
		\begin{align*}
		 \limsup_{n\to\infty}\frac{\lm_{n}}{D_{n}}\leq  {\frac{2+l}{1+l}}.
		\end{align*}
		Note that in the situation $D_\infty=+\infty$and $q\in \mathcal{K}_{\al}^{{\theta}}$, one has $l ? ?\alpha$. 
		 \end{rem}


\subsection{Magnetic isoperimetry}\label{s:proof}

We prove the main theorem via isoperimetric techniques.  For a
function $f:X\to\R$ one defines the level sets
\begin{align*}
\Omega_{t}(f){:=}\{x\in X\mid f(x)>t\},\quad t\in\R.
\end{align*}
In the non-magnetic case the following area and coarea are well-known, see e.g. \cite[Theorems~12~and~13]{KL2}
\begin{align*}
\int_{0}^{\infty}m(\Omega_{t}(f))dt&=\sum_{x\in X}f(x)m(x),\\
\int_{0}^{\infty}b(\partial \Omega_{t}(f))dt&=\frac{1}{2}\sum_{x,y\in X}
b(x,y)|f(x)-f(y)|,
\end{align*}
for all $f:X\to[0,\infty)$. The key ingredient of the proof of our
main theorem is the following magnetic co-area inequality.

\begin{lemma}\label{l:coarea}For all $f\in C_{c}(X)$
	\begin{align*}
	\int_{0}^{\infty}\Big(\iota_{{1,\theta}}&(\Omega_{t}(|f|^{{p}})
	+\frac{1}{2}b(\partial \Omega_{t}(|f|^{{p}}))\Big)dt\\
	&
	\leq{\frac{\sqrt{p^2+1}}{2}}\sum_{x,y\in X}
	b(x,y)|f(x)-e^{i\theta(x,y)}f(y)|\left(\frac{|f(x)|^{{p}}+|f(y)|^{{p}}}2\right)^{\frac{p-1}p}.
	\end{align*}
\end{lemma}
\begin{proof} For finite graphs  and $ p=2 $, this  formula  has been shown in \cite[Lemmas~4.3 and 4.7]{LLPP}, which can also be extracted from the proof of \cite[Lemma~3.2]{BSS}.
	The ideas carries over directly to our setting. 
	
	For a given function $f:X\to \C$, we define the following complex valued function:
	\begin{equation*}
	F_f(x,t):=\left\{
	\begin{array}{ll}
	f(x)/|f(x)| &: \hbox{if $x\in \Omega_t(|f|^{{p}})$,} \\
	0 &: \hbox{otherwise.}
	\end{array}
	\right.
	\end{equation*}
	Then, we can calculate \begin{align*}
	\int_{0}^{\infty}\Big(\iota_{1,\theta}(\Omega_{t}(|f|^{{p}})
	+&\frac{1}{2}b(\partial \Omega_{t}(|f|^{{p}}))\Big)dt\\
	\leq &\int_0^{\infty}\frac{1}{2}\sum_{x,y\in X}b(x,y)|F_f(x,t)-e^{i\theta(x,y)}F_f(y,t)|dt\\
	=&\frac{1}{2}\sum_{x,y\in X}
	b(x,y)\int_0^\infty|F_f(x,t)-e^{i\theta(x,y)}F_f(y,t)|dt.
	\end{align*}
	Note in the last equality above, we used Tonelli's theorem.
	
	For two vertices $x,y\in X$, we assume w.l.o.g.\ that $|f(x)|\leq
	|f(y)|$. We calculate
	\begin{align*}
	 \int_0^\infty|F_f(x,t)-e^{i\theta(x,y)}F_f(y,t)|dt&=\int_0^{|f(x)|^{{p}}}\left|\frac{f(x)}{|f(x)|}-\frac{e^{i\theta(x,y)}f(y)}{|f(y)|}\right|dt+\int_{|f(x)|^{{p}}}^{|f(y)|^{{p}}}1dt\notag\\
	=&\left|\frac{f(x)}{|f(x)|}-\frac{e^{i\theta(x,y)}f(y)}{|f(y)|}\right||f(x)|^{{p}}+|f(y)|^{{p}}-|f(x)|^{{p}}.
	\end{align*}
	Now, we apply Lemma~\ref{l:unit vectors beta^p} below with
		\begin{align*}
		v=\frac {f(x)} {|f(x)|}, \qquad w=\frac{e^{i\theta(x,y)}f(y)}{|f(y)|}, \qquad \mbox{ and }\quad \beta=\frac{|f(y)|}{|f(x)|}
		\end{align*}
		and obtain
	\begin{align*}
	\int_0^\infty|F_f(x,t)-e^{i\theta(x,y)}&F_f(y,t)|dt \\\leq
	&{\sqrt{p^2+1}}|f(x)-e^{i\theta(x,y)}f(y)|\left(\frac{|f(x)|^{{p}}+|f(y)|^{{p}}}2\right)^{\frac{p-1}p}.\label{eq:coarea2}
	\end{align*}
	Combining this with the estimate in the beginning,  the statement of the lemma
	follows.
\end{proof}

\begin{lemma}\label{l:unit vectors beta^p}
	For all $v,w \in \T$ and for all $\beta,p\in[1,\infty)$ one has
	\begin{align*}
	|v-w| + \beta^p - 1 \leq
	\sqrt {p^2+1} |v- \beta w|\left(\frac{1+\beta^{p}}2\right)^{\frac{p-1}{p}}.
	\end{align*}
	
\end{lemma}
\begin{proof}
	The claim is obvious for $\beta=1$. Thus, we can assume $\beta>1$.
	Similarly to \cite[Proposition~A.1]{BSS}, we set $t:=|v-w|$ and $s:=t/(\beta^p-1)$. We obtain due to $|v|=|w|=1$
	\begin{align*}
	|v-\beta w| &= \sqrt{(\beta-1)^2 + \beta  t^2 }\geq \sqrt{(\beta-1)^2+t^2}\geq (\beta-1)\sqrt{1+s^2p^2},
	\end{align*}
	where we used $ |\beta^{p}-1|\ge p|\beta-1| $ in the last estimate.
	Moreover,
	\begin{align*}
	\sqrt{p^2+1}\sqrt{1+s^2p^2} \geq \sqrt{p^2(1+s^2) +  2p^2 s}= (1+s)p
	\end{align*}
	which yields together with the estimate above
	\begin{align*}
	\sqrt {p^2+1} |v- \beta w|\ge (1+s)p(\beta-1).
	\end{align*}
	Furthermore, due to S. Amghibech, \cite[Lemma~3]{Amg03}, see also \cite[Lemma~3.8]{KM}, we have
	\begin{align*}
	p(\beta-1)\left(\frac{\beta^p+1}2\right)^{\frac{p-1}p}\ge 	\beta^p -1 .
	\end{align*}
	Hence, taking the last two inequalities together
	we obtain the desired estimate
	\begin{align*}
	\sqrt {p^2+1} |v- \beta w|\left(\frac{1+\beta^{p}}2\right)^{\frac{p-1}{p}}
	\ge(1+s)(	\beta^p -1)
	=		|v-w| + \beta^p - 1.
	\end{align*}
	This finishes the proof.
\end{proof}

\begin{lemma}\label{l:i2iia}
	{Let $p\in[1,\infty)$}.
	If a magnetic graph is $(a,k)_1$-magnetic sparse, $a,k\ge0$, and
	$q\ge0$, then
		\begin{align*}
		(1-\tilde a)\langle  \Deg,|f|^p\rangle-\tilde k \|f\|_p^p \leq Q_{p,\theta}(f),\qquad  f \in C_{c}(X),
		\end{align*}
	where
		\begin{align*}
		\tilde a=1-\frac{2^{p-1}}{(p^2+1)^{\frac p 2}(1+a)^p},\qquad \tilde k=\frac{2^{p-1}pk}{(p^2+1)^{\frac p 2}(1+a)^p}.
		\end{align*}
\end{lemma}
\begin{proof}
	Let $f\in C_c(X)$. If
			\begin{equation*}
		\langle  \Deg,|f|^p\rangle<k  \Vert f\Vert_p^p,
		\end{equation*}
	then {the announced values of $\tilde a$ and $\tilde k$ work in this case since}
	\begin{equation*}
	\tilde{k}\geq (1-\tilde{a})k.
	\end{equation*}
	So, we assume {$\langle  \Deg,|f|^p\rangle \geq k  \Vert f\Vert_p^p$}.
	Note that the following area formula holds,
	\begin{multline*}
	\int_0^{\infty}\left(b(\Omega_t(|f|^{{p}})\times \Omega_t(|f|^{{p}}))+\frac{1}{2}b\left(\partial \Omega_t(|f|^{{p}})\right)+(qm)(\Omega_t(|f|^{{p}})\right)dt
	\\=\int_0^{\infty}\sum_{x\in \Omega_t(|f|^{{p}})}\deg(x)dt=\langle \Deg , |f|^{p}\rangle.
	\end{multline*}
	Now, we calculate
	\begin{align*}
	&\langle \Deg , |f|^{p}\rangle-k\Vert f\Vert_{{p}}^{{p}}\\
	=&\int_0^{\infty}\left(b(\Omega_t(|f|^{{p}})\times \Omega_t(|f|^{{p}}))+\frac{1}{2}b(\partial \Omega_t(|f|^{{p}}))+(qm)(\Omega_t(|f|^{{p}})-km(\Omega_t(|f|^{{p}}))\right)dt\notag\\
	\leq &{\int_0^{\infty} (1+a) \left(\iota_{1,\theta}(\Omega_t(|f|^{{p}}))+\frac{1}{2}b(\partial\Omega_t(|f|^{{p}}))+(qm)(\Omega_t(|f|^{{p}})
	\right)dt.}
	\end{align*}

	In the above inequality, we used the $(a, k)_1$-magnetic-sparseness.
	Applying Lemma~\ref{l:coarea} and {H\"older} inequality, we
	obtain {with  $C={(1+a) \sqrt{p^2+1}}/{2}$ } that 
	\begin{align*}
	&\langle \Deg , |f|^{p}\rangle-k\Vert f\Vert_{{p}}^{{p}}\\
	\leq & (1+a) \Biggl(\frac {\sqrt{p^2+1}} {2}\sum_{x,y\in X}b(x,y)|f(x)-e^{i\theta(x,y)}f(y)|\left(\frac{|f(x)|^{{p}}+|f(y)|^{{p}}}2\right)^{\frac{1}{p^{*}}}\\
	&\qquad\qquad\qquad+\sum_{x\in X}(|f|^{p}qm)(x)\Biggr)\\
	\leq &  (1+a)  \left(  \frac{1}{2} \sum_{x,y\in X}b(x,y)|f(x)-e^{i\theta(x,y)}f(y)|^{{p}}+\sum_{x\in X}|f(x)|^{{p}}(q m)(x)\right)^{\frac{1}{{{p}}}}\\
	&\hspace{.5cm}\times \left( \frac{(1+p^2)^\frac{p^*}{2}}{2} \sum_{x,y\in X}b(x,y)\left(\frac{|f(x)|^{{p}}+|f(y)|^{{p}}}2\right)^{\frac{1}{p^*}\cdot {p}^{*}}+\sum_{x\in X}|f(x)|^{{p}}(qm)(x)\right)^{\frac{1}{{p}^{*}}}\\
	 \leq &2^{\frac{1}p - 1} \sqrt{p^2+1}(1+a)Q_{{p},\theta}(f)^{\frac{1}{{p}}}\langle  \Deg,|f|^{{p}} \rangle^{\frac{1}{p^{*}}}.
	\end{align*}
	Since the left hand side of the above inequality is non-negative by assumption, we can take {$p$-th power} on both sides. Therefore, we arrive at
		\begin{align*}
		\langle  \Deg,|f|^p\rangle^p - kp \langle  \Deg,|f|^p\rangle^{p-1}\|f\|_p^p &\leq
		\left(\langle \Deg,|f|^p \rangle - k \|f\|_p^p \right)^p \\& \leq
		2^{1-p}(p^2+1)^{\frac p 2}(1+a)^p Q_{p,\theta}(f) \langle  \Deg,|f|^p\rangle^{\frac {p} {p^{*}}}.
		\end{align*}
	This implies {due to $p/p^{*} = p-1$} that
		\begin{equation*}
		\frac{2^{p-1}}{(p^2+1)^{\frac p 2}(1+a)^p}\langle  \Deg,|f|^p\rangle-\frac{2^{p-1}pk}{(p^2+1)^{\frac p 2}(1+a)^p}\Vert f\Vert_p^p\leq Q_{p,\theta}(f).
		\end{equation*}
This shows the statement with the choice of $(\tilde a, \tilde k)$ in
	the statement of the lemma.
\end{proof}
\begin{lemma}\label{l:i2ii} If the magnetic graph is $(a,k)_1$-magnetic sparse, $a,k\ge0$, and
	$q\in\mathcal{K}_{\al}^{{p},\theta}$ then
	\begin{align*}
	(1-\tilde a)\langle \Deg, {|f|^p}\rangle -\tilde k {\|f\|_p^p} \leq Q_{p,\theta} (f),\qquad f \in C_{c}(X).
	\end{align*}
	where
	\begin{align*}
	\tilde a&=1-{\frac{1- \al}{2^{1-p}(p^2+1)^{p/2}(1+a)^p-\al}}, \\
	\tilde
	k&=\frac{(1-\al)kp+(2^{1-p}(p^2+1)^{p/2}(1+a)^p-1)C_{\al}}{2^{1-p}(p^2+1)^{p/2}(1+a)^p-\al},		
	\end{align*}
	and $C_{\al}$ is the constant from the  bound $\langle q_{-}, {|f|^p} \rangle\leq \al
	Q_{{p},b,\theta,q_{+},m}+ C_{\al} {\|f\|_p^p}$.
\end{lemma}
\begin{proof}
	By the lemma above we have, for all $ f\in C_{c}(X) $, the inequality    $ (1-\tilde
	a_{0})\langle  (\Deg+q_{-}),|f|^p\rangle-\tilde k_{0}\|f\|_p^p\le Q_{{p},b,\theta,q_{+},m}(f)$ with
	constants $\tilde a_{0}$ and $\tilde k_{0}$ . Hence, together with
	the  bound for $q_{-}$ we obtain
	\begin{align*}
	\frac{(1-\al)(1-\tilde a_{0})}{1-\al(1-\tilde a_{0})}
	\langle  \Deg,|f|^p\rangle-\frac{(1-\al)\tilde k_{0}+\tilde a_{0}C_{\al}}{{1-\al(1-\tilde
			a_{0})}}{\|f\|_p^p}   \le Q_{{p},b,\theta,q,m}(f).
	\end{align*}
	by a straightforward calculation (similar to \cite[Lemma~A.3]{BGK}).
	With the specific constants of Lemma~\ref{l:i2iia}, the statement
	follows.
\end{proof}

\begin{lemma}\label{l:ii2i} {Let $(b,\theta,q,m)$ be a magnetic graph with $q\geq 0$.}
Let $p\in[1,\infty)$.
		If {for some $0<\tilde a <1$}  the magnetic graph  satisfies
		\begin{align*}
		(1-\tilde a)\langle  \Deg,|f|^p\rangle-\tilde k \|f\|_p^p  \leq Q_{p,\theta}(f),\qquad f\in C_{c}(X),
		\end{align*}
		then the graph is $(a,k)_{{p}}$-magnetic sparse with
		{
		\begin{align*}
		a=\frac{\tilde a}{1-\tilde a} \quad   k=\frac{\tilde k}{1-\tilde a}.
		\end{align*}
		}
\end{lemma}
\begin{proof}
	Let $W\subseteq X$ be a finite
	set and
	let $\tau_0: W\to \mathbb{T}$ be the function that attains the minimum in the definition of $\iota_{p,\theta}(W)$. We define the following function:
	\begin{equation*}
	f_0(u):=\left\{
	\begin{array}{ll}
	\tau_0(u) &: \hbox{if $u\in W$,} \\
	0 &: \hbox{otherwise.}
	\end{array}
	\right.
	\end{equation*}
	We calculate
		\begin{align*}
		Q_{p,\theta}(f_0)=&\frac{1}{2}\sum_{x,y\in X}b(x,y)|f_0(x)-e^{i\theta(x,y)}f_0(y)|^p+\sum_{x\in X}q(x)m(x)|f_0(x)|^p\notag\\
		=&\frac{1}{2}\sum_{x,y\in W}|\tau_0(x)-e^{i\theta(x,y)}\tau_0(y)|^p+\frac{1}{2}b(\partial W)+(qm)(W)\notag\\
		{=} &\iota_{p,\theta}(W)+\frac{1}{2}b(\partial W)+(qm)(W),
		\end{align*}
	Applying the assumed
	inequality to the function $f_0$, we obtain
	\begin{equation*}
	(1-\tilde{a})\left(b(W\times W)+\frac{1}{2}b(\partial W)+(qm)(W)\right)-\tilde k m(W)\leq \iota_{{p},\theta}(W)+\frac{1}{2}b(\partial W)+(qm)(W).
	\end{equation*}
	Therefore, we have
	\begin{equation*}
	(1-\tilde{a})b(W\times W)-\tilde{k}m(W)\leq \iota_{{p},\theta}(W)+\tilde{a}\left(\frac{1}{2}b(\partial W)+(qm)(W)\right).
	\end{equation*}
	This proves the lemma.
\end{proof}

To prove Theorem~\ref{main} we apply the  lemmas
above and the closed graph theorem.

\begin{proof}[Proof of Theorem~\ref{main}]
{First we prove the theorem for non-negative potential $ q\ge0 $ in which case $ \deg=\deg^{(+)}:=\deg+q_{-} $.}
 Let $p\geq 1$ and let us denote $\text{(i')}_{p}$ and $\text{(ii)}_{p}$ the statement (i') and (ii) with this $p$.  \\
	$   \text{(i)} \Rightarrow$ $\text{(ii)}_{p}$: This follows from Lemma \ref{l:i2iia}.\\
	$\text{(ii)}_{p}$ $\Rightarrow   \text{(i')}_p$: This follows from Lemma~\ref{l:ii2i}.\\
	$\text{(i')}_{p} \Rightarrow  \text{(i)}$: This follows by Remark \ref{r:triv}.\\
{We now consider the case $p=2$}.\\
	$\text{(ii)}_{2}$ $\Rightarrow$ (iii): By definition $D(Q_{\theta})\subseteq
	\ell^{2}(X,m)$. The inequality in $\text(ii)_2$ implies $  D(Q_{\theta}) \subseteq  \ell^{2}(X,\deg)
	$ as $\langle{\Deg f},{f
	}\rangle_{m}=\sum_{X}\Deg  |f|^{2}m=\sum_{X}\deg |f|^{2}$. The
	inclusion $ \ell^{2}(X,\deg)  \cap \ell^2 (X,m) \subseteq D(Q_{\theta}) $ follows
	from the inequality $Q_{\theta}\leq 2\Deg $ which holds true because $ q\ge0 $.
	(iii) $\Rightarrow$ $\text{(ii)}_{p=2}$: This follows from the closed graph theorem
	applied to the embedding $j:D(Q_{\theta})\to \ell^{2}(X,\deg^{(+)}+m)$
	(cf. \cite[Theorem~A.1]{BGK}).

{We now turn to the case where  $q\in\mathcal{K}_{\al}^{{p},\theta} $}.
From the definition of magnetic sparseness $q_-$ does not appear in (i) and (i'). 
 If (ii) holds for  $Q_{\theta,q}$, then it holds for $Q_{\theta,q_+}$  with the same constants.  The reciprocal is true with a change in the constants using
 the assumption $  \langle q_{-},|f|^p\rangle \leq \al   Q^{(\mathrm{comp})}_{p,b,\theta,q_{+},m}{(f)}+C_{\al} \|f\|_p^p, $
which is  the definition of   $\mathcal{K}_{\al}^{{p},\theta} $.  In the case  $p=2$ and $q\in\mathcal{K}_{\al}^{{2},\theta} $,  the domains of $Q_{\theta,q}$ and $Q_{\theta,q_+}$ are the same.
 As a consequence, each of the assertion $\text{(i)} ,\text{(i')}_{p},\text{(ii)}_{p}$,  and $\text{(iii)}$ in the case $q\in\mathcal{K}_{\al}^{{2},\theta}$ holds for $(b,\theta,q,m)$ if and only if it  holds for $(b,\theta,q_+,m)$.
 The theorem for $q\in\mathcal{K}_{\al}^{{p},\theta} $ follows.

The value of the constants  $\tilde a, \tilde k$ are given by Lemma~\ref{l:i2ii}.
\end{proof}


\section{Examples of magnetic sparseness}\label{s:examples}
In this section we consider products of graphs, magnetic cycles and tessellations. In the section of products of graphs we provide a structural description of part of the results of \cite{GT}. In the section of magnetic cycles we compute the frustration index for magnetic cycles for $ p=1,2 $. Finally, we use these results to conclude magnetic sparseness for tessellations under the assumption that the magnetic strength of cycles within the tessellation is large with respect to their length.

\subsection{Products of graphs}\label{s:prod}

In this section we apply our results to cartesian products of
graphs. Let two magnetic graphs $(b_{1},\theta_{1},q_{1},m_{1})$
over $X_{1}$ and $(b_{2},\theta_{2},q_{2},m_{2})$ over $X_{2}$ together with their magnetic forms $ Q_{\theta_{1}}=	Q_{2,b_{1},\theta_{1},q_{1},m_{1}} $ and  $ Q_{\theta_{2}}=	Q_{2,b_{2},\theta_{2},q_{2},m_{2}} $  be given.
Furthermore let $\mu:X_{1}\times X_{2}\to{(0,\infty)}$.

We define the product
$(b,\theta,q,m_{\mu})=(b_{1},\theta_{1},q_{1},m_{1})\otimes(b_{2},\theta_{2},q_{2},m_{2})$
 with respect to $\mu$ via
\begin{align*}
b(x,y)&=b_{1}(x_{1},y_{1}) m_{2}(x_{2})
1_{\{x_{2}=y_{2}\}}+b_{2}(x_{2},y_{2})m_{1}(x_{1}) 1_{\{x_{1}=y_{1}\}},\\
\theta(x,y)&=\theta_{1}(x_{1},y_{1})
1_{\{x_{2}=y_{2}\}}+\theta_{2}(x_{2},y_{2})
1_{\{x_{1}=y_{1}\}},\\
q(x)&=\frac{1}{\mu(x_{1},x_{2})}\left(q_{1}(x_{1})+q_{2}(x_{2})\right),\\
m_{\mu}(x)&=\mu(x_{1},x_{2})m_{1}(x_{1})m_{2}(x_{2});
\end{align*}
for $x=(x_{1},x_{2}), y=(y_1,y_2) \in X_1 \times X_2$.
Note that for all $x=(x_{1},x_{2})\in X_{1}\times X_{2}$,
{ 
\begin{align*}
\sum_{y\in X_{1}\times X_{2}}b(x,y)=m_{1}(x_{1})\sum_{y_{2}\in
	Y_{2}}b_{2}(x_{2},y_{2})+m_{2}(x_{2})\sum_{y_{1}\in
	Y_{1}}b_{1}(x_{1},y_{1})<\infty.
\end{align*}
}

This product is a natural choice as the following lemma shows.

\begin{lemma}\label{l:prod} The quadratic form $Q_{\theta}=Q_{2,b,\theta,q,m_{\mu}}$ acts as
	\begin{align*}
	Q_{\theta}(f)=
	\sum_{x_{2}\in X_{2}}
	Q_{\theta_{1}}(f(\cdot,x_{2}))m_{2}(x_{2})+
	\sum_{x_{1}\in X_{1}}
	Q_{\theta_{2}}(f(x_{1},\cdot))m_{1}(x_{1})
	\end{align*}
	for $f\in
	D(Q_{\theta})\subseteq \ell^{2}(X_{1}\times X_{2},m_{\mu})$ and the corresponding selfadjoint Laplacian $H_\theta=H_{b,\theta,q,m_\mu}$ is a restriction
	of the operator
	$\mathcal{H}_{\theta}=\mathcal{H}_{b,\theta,q,m_{\mu}}$ acting as
	\begin{align*}
	\mathcal{H}_{\theta}=\frac{1}{\mu}[1_{X_{1}}\otimes \mathcal{H}_{\theta_{2}}+
	\mathcal{H}_{\theta_{1}}\otimes 1_{X_{2}}].
	\end{align*}
\end{lemma}
\begin{proof}This follows by direct calculation.
\end{proof}

There are several ways to prove that the essential spectrum of the magnetic Laplacian is empty, e.g., \cite{CdVTHT} and \cite{GT}. Here we use Lemma~\ref{l:prod} and the sparseness of only one of the graphs in order to prove the discreteness of the spectrum of $H_\theta$. That is the spirit of the techniques developed in \cite{GT} (where the authors use a
	slightly different product).

\begin{thm}
Let $(b_{1},\theta_{1},q_{1},m_{1})$ be a magnetic graph over $X_1$ and $q_1\geq 0$. Let $(b_{2},\theta_{2},q_{2},m_{2})$ be a $(a,0)_1$-magnetic sparse graph over  $X_2$, with  $q_2\geq 0$ and  $\inf \Deg_2 >0$. Let $\mu:X_1\times X_2 \to (0, \infty)$ be such that
	\[{\mu(x_1, x_2) \to 0, }\]
as $(x_1, x_2)$ leaves every compact set of $X_1\times X_2$.
		Take $(b,\theta,q,m_{\mu}):=(b_{1},\theta_{1},q_{1},m_{1})\otimes(b_{2},\theta_{2},q_{2},m_{2})$, constructed as above. Then  $H_\theta$ has purely discrete spectrum.
\end{thm}
If $X_2$ is finite and connected and $q_2\geq 0$, note that $(b_{2},\theta_{2},q_{2},m_{2})$ is $(a,0)_1$-magnetic sparse
if $\iota_{1,\theta_{2}}(X_{2})>0$ or if $q_2(X_2)>0$.

\begin{proof}Let $0\blue{<} D_{2}\leq \mathrm{Deg}_{2}$, where $ \Deg_{2} $ is the weighted vertex degree of 	$(b_{2},\theta_{2},q_{2},m_{2})$.
	Using Lemma~\ref{l:prod} and Theorem~\ref{main} for
	$(b_{2},\theta_{2},q_{2},m_{2})$
	we
	infer for all $f\in C_{c}(X)$
	\begin{align*}
	Q_{\theta}(f)&\ge \sum_{x_{1}\in
		X_{1}}Q_{\theta_{2}}(f(x_{1},\cdot))m_{1}(x_{1})\\
	&\ge(1-\tilde a)\sum_{x_{1}\in
		X_{1}}\langle{\Deg_{2}},{ |f(x_{1},\cdot)|^{2}}\rangle_{m_{2}} m_{1}(x_{1})\\
	&\geq (1-\tilde a)D_{2}\langle{\frac{1}{\mu}f},{f }\rangle_{m_{\mu}},
	\end{align*}
	with $\tilde a<1$. The discreteness of spectrum of {$H_\theta$} follows from the
	Min-Max-Principle and the fact that $\mu$ tends to zero as
	leaving every compact set.
\end{proof}

We now compute the magnetic-sparseness constants and the  frustration indices of  products of graphs.

\begin{thm}\label{thm:mag_sparseness_product} {Let $\mu:X_1\times X_2 \to (0,\infty)$. Let $a,k\geq 0$ and $p\in[1,\infty)$.
		Suppose that  $(b_{1},\theta_{1},q_{1},m_{1})$
		and $(b_{2},\theta_{2},q_{2},m_{2})$  are $(a,k)_p$-magnetic-sparse.}
	\begin{enumerate}
		\item If $k=0$ then  $(b,\theta,q,m_{\mu})$ is $(a,0)_p$-magnetic-sparse.
		\item If $\mu\geq c>0$ for some $c>0$, then $(b,\theta,q,m_{\mu})$ is $(a, k_\mu)_{p}$-magnetic-sparse where $k_\mu:= k/c$.
	\end{enumerate}
\end{thm}

In order to prove Theorem \ref{thm:mag_sparseness_product}, we
show a lemma, which is interesting in its own right.
\begin{lemma}\label{l:frustration_index_product}Let $W\subseteq X_{1}\times X_{2}$ and {$p\in[1,\infty)$}. For $x_{1}\in X_{1}$, we denote
	$W_{x_{1}}=\{y\in X_{2}\mid (x_{1},y)\in W\}$ and for
	$x_{2}\in X_{2}$ we denote $W_{x_{2}}=\{x\in X_{1}\mid
	(x,x_{2})\in W \}$. Then, we have
	\begin{equation*}
	\iota_{p,\theta}(W)=\sum_{x_1\in X_1}m_1(x_1)\iota_{p,\theta_2}(W_{x_1})+\sum_{x_2\in X_2}m_2(x_2)\iota_{p,\theta_1}(W_{x_2}).
	\end{equation*}
\end{lemma}
\begin{proof}Let $ p\in[1,\infty) $ be fixed.
	Let $\tau_0: W\to \mathbb{T}$ be the function that attains the minimum in the definition of $\iota_{p,\theta}(W)$. By definition, we have
	\begin{align*}
	\iota_{p,\theta}(W)=&\frac{1}{2} \sum_{x_1\in X_1}m_1(x_1)\sum_{x_2,y_2\in W_{x_1}}b_2(x_2,y_2)|\tau_0(x_1,x_2)-e^{i\theta_2(x_2,y_2)}\tau_0(x_1, y_2)|^{p}\\
	&+\frac{1}{2} \sum_{x_2\in X_2}m_2(x_2)\sum_{x_1,y_1\in W_{x_2}}b_1(x_1,y_1)|\tau_0(x_1,x_2)-e^{i\theta_1(x_1,y_1)}\tau_0(y_1,x_2)|^{p}\\
	\geq & \sum_{x_1\in X_1}m_1(x_1)\iota_{p,\theta_2}(W_{x_1})+\sum_{x_2\in X_2}m_2(x_2)\iota_{p,\theta_1}(W_{x_2}).
	\end{align*}
	On the other hand, letting $\tau_1: W_{x_{2}}\to \mathbb{T}$ and $\tau_2: W_{x_{1}}\to \mathbb{T}$ be the function that attains the minima in the definitions of $\iota_{p,\theta_1}(W_{x_{2}})$ and $\iota_{p,\theta_2}(W_{x_{1}})$, respectively, we have
	\begin{align*}
	\sum_{x_1\in X_1}&m_1(x_1)\iota_{p,\theta_2}(W_{x_1})+\sum_{x_2\in X_2}m_2(x_2)\iota_{p,\theta_1}(W_{x_2})\\
	=& \frac{1}{2} \sum_{x_1\in X_1}m_1(x_1)\sum_{x_2,y_2\in W_{x_1}}b_2(x_2,y_2)|\tau_1(x_1)\tau_2(x_2)-e^{i\theta_2(x_2,y_2)}\tau_1(x_1)\tau_2(y_2)|^{p}\\
	&+\frac{1}{2} \sum_{x_2\in X_2}m_2(x_2)\sum_{x_1,y_1\in W_{x_2}}b_1(x_1,y_1)|\tau_1(x_1)\tau_2(x_2)-e^{i\theta_1(x_1,y_1)}\tau_1(y_1)\tau_2(x_2)|^{p}\\
	\geq &\iota_{p,\theta}(W),
	\end{align*}
	where we used the fact that $\tau_1\tau_2:
	(x_{1},x_{2})\mapsto\tau_1(x_1)\tau_2(x_2)$ is a map $ W\to
	\mathbb{T}$.
\end{proof}
\begin{proof}[Proof of Theorem \ref{thm:mag_sparseness_product}]
	Let $p\in[1,\infty)$ and let $W, W_{x_1}, W_{x_2}$ be as in Lemma~\ref{l:frustration_index_product}.
	Furthermore, for  $U\subseteq W_{x_{j}}$, let
{
	\begin{align*}
	R_{j}(U):=(1+a)\iota_{p,\theta_{j}}(U)+ a \left(\frac{1}{2}b_{j}(\partial
	U)+(q_{j,+}m_{j})(U)\right),\quad j=1,2.
	\end{align*}
	}
	By direct calculation using the $(a,k)_p$-magnetic-sparseness we
	obtain{
	\begin{align*}
	\lefteqn{   b(W\times W)=\sum_{x_{1}\in X_{1}}m_{1}(x_{1})
		b_{2}(W_{x_{1}}\times W_{x_{1}})
		+\sum_{x_{2}\in X_{2}}m_{2}(x_{2})b_{1}(W_{x_{2}}\times
		W_{x_{2}})}\\
	&\leq\sum_{x_{1}\in
		X_{1}}m_{1}(x_{1})(R_{2}(W_{x_{1}})+km_2(W_{x_{1}}))
	+\sum_{x_{2}\in X_{2}}m_{2}(x_{2})(R_{1}(W_{x_{2}})+km_1(W_{x_{2}}))
	\\
	&= (1+a)\iota_{p,\theta}(W)+ a \left(\frac{1}{2}b(\partial
	W)+(q_{+}m_\mu)(W)\right)+km(W).
	\end{align*}
	}
	Invoking  $m=m_{\mu}/\mu$ the statement follows by the assumption
	$k/\mu\leq k_{\mu}$.
\end{proof}
\begin{rem}
	Instead of cartesian products one can consider also subcartesian products. The considerations are almost identical.
\end{rem}

\subsection{Magnetic cycles}\label{s:cycle}
We study the notion of frustration indices and of magnetic-sparseness in the case of a cycle. We start with a definition.

\begin{defi}[Magnetic cycle]
(a)	We call a magnetic graph  ${\Cc:=}(b,\theta,0,m)$ over a finite set  $Y$  a \emph{cycle (graph)} of \emph{length} $$ l(\Cc) =n, $$ if     there is a bijective map $\Phi:
	Y \ \to \{0,\ldots,n-1\}$ such that  $b(x,y){>0}$ if and only if
	$(\Phi(x)-\Phi(y) \mod n)\in {\{1,n-1\}} $. We set $x_j:=\Phi^{-1}(j)$, $ j=1,\ldots,n-1 $ and $x_{n}:=x_0$.\\
	
(b) The \emph{magnetic flux} of a cycle $\Cc$   is defined as
	\[F_{\theta}(\Cc):= \sum_{j=0}^{n-1} \theta(x_j, x_{j+1}) \mod (2\pi).\]

(c) The \emph{{strength of the magnetic field}} of a cycle $\Cc$  is defined as  $$ s_{\theta}(\Cc) :=\left| 1 - \exp\left(i F_\theta(\Cc) \right) \right|. $$
\end{defi}
{Note that while the sign of $ F_{\theta}(\Cc)$ still depends on the choice of $ \Phi $,   the value of $ s_{\theta}(\Cc) $ is independent of  $ \Phi. $
}

We turn to the computation of the frustration indices.

\begin{pro}[Frustration indices of a cycle]\label{eg:circle}
	Let  $\Cc$  be a magnetic cycleover $ Y $ of length $n$ with standard edge weights and let $W \subseteq Y$ be finite. Then,
	\begin{itemize}
		\item[(a)]
		\begin{align*}
		\iota_{1,\theta}(W)  = \begin{cases}
		s_\theta(\Cc)&:  W = Y, \\
		0&: W\neq Y.
		\end{cases}
		\end{align*}
		\item[(b)]
		\begin{align*}
		\iota_{2,\theta}(W)  = \begin{cases}
		n |1-e^{i \delta/n}|^2&:  W = Y, \\
		0&: W \neq Y,
		\end{cases}
		\end{align*}
		where $\delta:=\min_{k\in \Z}|F_\theta(\Cc) - 2k\pi|$.
	\end{itemize}
	
\end{pro}
\begin{proof} 
For (a) and (b) it is enough to consider with $W=Y$ by  Proposition~\ref{p:iota}~(d).
	
	(a) This follows easily from \cite[Theorem 4.10]{LMP}. There it is proven that $ \iota_{1,\theta}(Y) $ is assumed for  a function $ \tau $ that is supported on a spanning tree and that satisfies $ \tau(x) = e^{i\theta(x,y)}\tau(y) $ for neighbors $ x $ and $ y $ on this spanning tree.

	(b) In \cite[Lemma 2.3]{CdVTHT} the bottom of  $\sigma(H_\theta)$ is computed for $m=1$ to be $ |1-e^{i \delta/n}|^2 $. Since the eigenfunctions have constant absolute value (say, equals to $1$), the minimizer of the Rayleigh quotient minimizes also $\iota_{2, \theta}(X)$ and, thus,
	\[ \iota_{2, \theta}(Y) = n  \inf \sigma(H_\theta)=n|1-e^{i \delta/n}|^2.\]
	
\end{proof}

\begin{rem}\label{r:iota}(a) The minimizers of $\iota_{1, \theta}$ and of $\iota_{2, \theta}$ are very different. For $\iota_{2, \theta}$, all  edges have the same contribution whereas the contribution for $\iota_{1, \theta}$ is concentrated solely on one edge.
	
	(b) Note that $\iota_{1, \theta}$ is bounded by $2$ and is independent of the length of the cycle. It depends only on the magnetic flux.

(c) In the case of a general  magnetic cycle $\Cc$ over $Y$ (not necessarly with standard egde weights), the same proof as above gives:
{
\[
\iota_{1,\theta}(Y)= \min_ {1\leq j \leq n} \left\{ b(x_{j},x_{j+1})\right\} \, |1-e^{iF_{\theta}(\mathcal{C})}|.
\]
}
{The exact value of $\iota_{2,\theta}(Y)$ in this situation is not clear}.
\end{rem}

We now turn to magnetic and bi-magnetic sparseness of cycles.
\begin{pro}\label{ex:notbi}
Let $ \Cc$ be a magnetic cycle  over $ Y $ of length $l(\Cc)$ with standard edge weights. Then,
  $\Cc$ is $\left(a,0 \right)_1$-magnetic-sparse for some $a>0$ if and only if  $ F_\theta(\Cc) \not  \equiv 0 \mod (2\pi)$.
 In this case one can take: $a= \frac{2l(\Cc)}{s_\theta(\Cc)}-1$.
 
Moreover, if $  l(\mathcal{C}) $ is even, then  the cycle $\Cc$ is $(a,0)_1$-magnetic-sparse for some $ a >0$ if and only if $ F_\theta(\Cc) \not  \equiv 0 \mod (2\pi)$; and 
if $ l(\mathcal{C}) $ is odd, then  the cycle $\Cc$ is $(a,0)_1$-magnetic-sparse for some $ a >0$ if and only if $ F_\theta(\Cc) \not  \equiv 0 \mod (\pi)$.
\end{pro}

\begin{proof}
{ First note that $ F_\theta(\Cc)  \not  \equiv 0 \mod (2\pi)$ is and only if $s_\theta(\Cc)>0$.
Assume first  that $s_\theta(\Cc)=0$, then $\iota_{1,\theta}(Y)= 0$ and \emph{magnetic sparseness inequality} \begin{align*}
 b(W\times W)\leq (1+a)\iota_{{p},\theta}(W)+a \big(\frac{1}{2}b(\partial
 W)+(q_{+}m)(W)\big) + k m(W)
\end{align*} can not hold for $W=Y$. The cycle $\Cc$ can not be  $(a,0)_1$-magnetic-sparse for some $ a >0$.
 Assume now $s_\theta(\Cc)>0.$ Set $a= \frac{2l(\Cc)}{s_\theta(\Cc)}-1$ and let $W\subseteq Y$.
  In the case $W=Y$, by  Proposition~\ref{eg:circle}~(a), we have  $b(Y\times Y) = 2l(\Cc)$,  $b(\partial W)=0$, and $\iota_{1,\theta}(W)=s_\theta(\Cc)$ and the magnetic sparseness inequality  holds for $W=Y$.  
	For $W\neq Y$, we have 
	$s_\theta(\Cc)\leq 2$, $b(W\times W) \leq 2l(\Cc)$, $b(\partial W)\geq 4$, and $\iota_{1,\theta}(W)= 0$ and magnetic sparseness  inequality  also holds for $W$. Thus $\Cc$ is $(a,0)_1$-magnetic-sparse.
	}
	
{
In the case when  $ l(\mathcal{C}) $ is even,  $\Cc$ is bi-partite and  by  Proposition \ref{p:bi-partite}, $\Cc$ is  $(a,0)$-bi-magnetic sparse if and only if it is $(a,0)$-magnetic sparse; the result  follows.
In the case when  $ l(\mathcal{C}) $ is odd,  one has: $F_{\theta+\pi}(\mathcal{C})\equiv F_{\theta}(\mathcal{C}) + \pi  \mod (2\pi)$ and the result follows.
}\end{proof}

\begin{rem}
Using Remark \ref{r:iota} (c), it can be shown that a similar statment holds in the case of general magnetic cycles (with non necessarly standard edge weights).
\end{rem}
\subsection{Subgraph criterion and tessellations}\label{s:subgraphs}

In this section, we give a useful criterion for magnetic sparseness using subgraphs.
Furthermore, we estimate the sparseness-constant for tessellations.
At the end, we show that regular triangulations with $\theta=\pi$ are magnetic sparse graphs, but not bi-magnetic sparse.

The results of this section are based on the following proposition,
where the subgraphs can be thought as cycles.

\begin{pro}[Subgraph criterion for magnetic sparseness]
	\label{pro:subgraph-criterion} Let $(b,\theta,q,m)$ be a magnetic
	graph over $X$ and let $p\in[1,\infty)$. Let $J$ be a set, $a,k\geq 0$,
	$C>c>0$, and $M>0$. Suppose
	$(b_j,\theta,q_{j},m_j)_{j \in J}$ is a family of
	$(a,k)_p$-magnetic-sparse-graphs such that for all $x,y \in X$,
\begin{itemize}
	\item [(a)] $ c \cdot b(x,y) \leq \sum\limits_{j \in J} b_j(x,y) \leq C \cdot
	b(x,y), $
	\item [(b)] $ \sum\limits_{j \in J} q_{j,+}(x)m(x) \leq C \cdot q_+(x)m(x) $,
	\item [(c)] $ 	\sum\limits_{j \in J}  m_j(x)    \leq M \cdot  m(x).
	 $
\end{itemize}	
	{Then, $(b,\theta,q,m)$ is $(aC/c + C/c-1,Mk/c)_p$-magnetic-sparse.} Moreover, if  $ k=0 $, then $ \mathrm{(a)} $ and $ \mathrm{(b)} $ are sufficient to conclude $(aC/c + C/c-1,0)_p$-magnetic-sparseness of $(b,\theta,q,m)$.
\end{pro}
\begin{proof}
	Let $W\subset X$ be finite and fix $p\in[1,\infty)$.
	We write $\iota_{p,\theta}^{(j)}$ for the frustration index of
	$(b_j,\theta,q_j,m_j)$.
	Since $(b_j,\theta,q_j,m_j)$ is $(a,k)_p$-magnetic sparse for all $j \in J$, we obtain
{\begin{align*}
	b(W\times W)  &\leq \frac 1 c \sum_{j\in J} b_{j}(W\times W) \\
	&\leq \frac 1 c \sum_{j\in J} \left((1+a)\iota_{p,\theta}^{(j)}(W)
	+ a \frac 1 2 b_j(\partial W)+ a (q_{j,+}m_j)(W)  + k m_j(W)
	\right)\\
	&\leq \frac{C}c \left((1+a)  \iota_{p,\theta}(W) + a \frac 1 2 b(\partial W)+a (q_{+}m)(W)\right) + \frac{ M k}{c} m(W).
	\end{align*}
	}
	This finishes the proof of the first the statement. The statement about $ k=0 $ is clear.
\end{proof}

Here, a \emph{tessellation} is a planar graph such that there exists a
set of subgraphs that are cycles such that every edge belongs to
exactly two cycles. (A subgraph is a restriction of the
corresponding maps to a subset of the space $X$.) For a more
restrictive notion of planar tessellations see e.g. \cite{BP1}.

We will apply Proposition \ref{pro:subgraph-criterion} to tessellations using the faces as
subgraphs. We show that every tessellation is magnetic sparse whenever the face
degree is upper-bounded and the magnetic strength of the faces is
lower-bounded from zero.

Let $\mathcal{F}$ be the set of faces of the graph. Let $F\in
\mathcal{F}$, we denote by $X_F$ the vertices which belong to $F$. We
define also $b_F:= b\cdot 1_{X_F\times X_F}$  and $\theta_F:=
\theta\cdot 1_{X_F\times X_F}$.
The graph  $\Cc_F:=(b_F,\theta_F,0, 1)$ over $X_F$ is a magnetic
cycle, see Section \ref{s:cycle}.

\begin{cor}[Magnetic sparseness of tessellations]
	Let a magnetic  tessellation $(b,\theta,0,1)$ over $ X $ with standard edge weights
 be given. If
{\[a=\sup_{F \in \mathcal F} \frac {2l(\Cc_F)}{s(\Cc_F)}-1 <\infty,\] }
	then the tessellation is $(a,0)_1$-magnetic-sparse.
\end{cor}
\begin{proof}
	Due to
	Proposition~\ref{eg:circle}, the graph over $\Cc_F$ is
{$\left(\frac{2l(\Cc_F)}{s(\Cc_F)}-1,0 \right)_1$}-magnetic-sparse.
	By the tessellation property, every
	edge belongs to exactly two cycles and, hence,
	\begin{align*}
	\sum_{F\in \mathcal{F}} b_F(x,y) = 2 b(x,y).
	\end{align*}
Thus, by Proposition~\ref{pro:subgraph-criterion}
we conclude the statement.
\end{proof}

We give the example of triangulation Cayley graphs $(b,\theta,q,m)$
which turn out to be magnetic sparse, but not bi-magnetic sparse. These triangulation Cayley graphs can be understood as a generalization of the triangulation of the plane by equilateral triangles.

\begin{eg}[Magnetic sparse, but not bi-magnetic sparse] \label{eg:not bi sparse}
	Let $(G,\cdot)$ be an {infinite abelian group which is finitely generated by at least two
		elements} with neutral
	element $e$ and symmetric generating set $S$ such that $e\not\in S$.
	{By choosing a possibly bigger $S$, we can furthermore assume that}
	 for every $s \in S$ there exists $r =r(s)\in S$ such
	that $rs \in S$. Because of the latter condition we refer to the
	corresponding Cayley graph as a \emph{triangulation}, i.e., every
	edge $\{g,sg\}$ in the corresponding Cayley graph belongs to at
	least one triangle namely $\{g,sg,rsg\}$ with $r$ chosen as above.
	Then,
	$$
	1\leq c:= \min_{s \in S}\# \{r \in S\mid rs \in S\}
	\leq \max_{s \in S}\# \{r \in S\mid rs \in S\} =:C < \infty.
	$$
	We take the magnetic potential $\theta = \pi$ and denote the Cayley
	graph by $(b,\theta,0,m)$ for some arbitrary measure $m$, potential
	$q=0$ and we take standard weights $b \in \{0,1\}$ with $b(g,h)=1$
	if and only if $gh^{-1}\in S$. Due to Proposition~\ref{eg:circle}, every
	triangle is {$(2,0)_1$} magnetic sparse. Since every edge is contained
	in at least $c$ and in at most $C$ triangles,
	Proposition~\ref{pro:subgraph-criterion} yields
	$(a,0)_1$-magnetic-sparseness of $(b,\theta,0,m)$ with
{
	$$
	a=\frac{3C}c -1 
	$$}
	On the other hand, $\theta + \pi = 0 \mod 2 \pi$ and, hence,
	$(b,\theta+\pi,0,m)=(b,0,0,m)$. It is well-known that abelian Cayley
	graphs {satisfy} $\inf_{W}\frac{b(\partial W)}{\# W}=0$.
	By $ \deg=\#S$ we infer  $\inf_{W}\frac{b(\partial W)}{b(W\times
		W)}=0$  and, therefore,  $(b,\theta+\pi,0,m)$ is not $(a,0)$
	magnetic sparse for any $a
	>0 $.
	If we choose $m$ such that $m(X) < \infty$ and if $G$ is infinite,
	we also have that $(b,\theta+\pi,0,m)$ is not $(a,k)$-magnetic sparse for any $a,k >0 $.
\end{eg}

\medskip

\textbf{Acknowledgement.} MK and FM acknowledge the financial
support of the German Science Foundation (DFG). SL acknowledges the financial support of the EPSRC Grant EP/K016687/1 "Topology, Geometry and Laplacians of Simplicial Complexes" and the hospitality of Universit\"at Potsdam during his visit in July 2016.

\end{document}